\documentclass[reqno,a4paper]{amsart}

\usepackage[english,activeacute]{babel}
\usepackage{amsmath}
\usepackage{amssymb}
\usepackage{graphicx}
\usepackage{amssymb,amsmath,amsthm,scalefnt}
\usepackage[T1]{fontenc}
\usepackage{url} 
\usepackage[latin1]{inputenc}
\usepackage{graphicx}
\usepackage{hyperref}
\usepackage{mathrsfs}
\usepackage{amsmath}
\usepackage{amsfonts}
\usepackage{graphics}
\usepackage{color}
\theoremstyle{plain}
\newtheorem{thm}{Theorem}[section]
\newtheorem{lemme}[thm]{Lemma}
\newtheorem{prop}[thm]{Proposition}
\newtheorem{cor}[thm]{Corollary}

\newtheorem{rmq}[thm]{Remark}
\numberwithin{equation}{section}

\newcommand{\Cb}  {{\mathbb C}}

\newcommand{\Rb}  {{\mathbb R}}
\newcommand{\Qb}  {{\mathbb Q}}

\DeclareMathOperator{\re}{Re}

\DeclareMathOperator{\sgn}{sgn}
\newcommand{\izi}{\int\limits_0^{\infty}}

\newcommand{\id}[2]{\int\limits_{#1}^{#2}}

\newcommand{\izt}{\int\limits_0^t}
\newcommand{\lr}{\left(}
\newcommand{\rr}{\right)}
\newcommand{\lc}{\left[}
\newcommand{\rc}{\right]}
\newcommand{\ef}  {{\mathbf e}}
\newcommand{\tx}[1][x]{T_{#1}}

\newcommand{\txe}[1][x]{T_{#1}^{\varepsilon}}
\newcommand{\txz}[1][x]{T_{#1}^{0}}

\newcommand{\sel}[1][\lambda]{{S_{\ef_{#1}}}}

\newcommand{\dE} {{\widehat{E}}}
\newcommand{\dP} {{\widehat{P}}}
\newcommand{\dPb} {{\widehat{\mathbb{P}}}}
\newcommand{\dQ} {{\widehat{Q}}}
\newcommand{\dR} {{\widehat{R}}}
\newcommand{\dX} {{\widehat{X}}}
\newcommand{\dc} {{\widehat{c}}}
\newcommand{\dq} {{\widehat{q}}}
\newcommand{\dphi} {{\widehat{\phi}}}
\newcommand{\dk} {{\widehat{k}}}
\newcommand{\dpm} {{{\widehat{p}}^{(m,1)}_1}}
\newcommand{\dpmt} {{{\widehat{p}}^{(m,t)}}}
\newcommand{\dpms} {{{\widehat{p}}^{(m,s)}}}
\newcommand{\dr} {{\widehat{r}}}
\newcommand{\dn} {{\widehat{n}}}
\newcommand{\dhp} {{\widehat{h}^{'}}}
\newcommand{\hd} {{\widehat{h}}}
\newcommand{\dL} {{\widehat{L}}}
\newcommand{\dtau} {{\widehat{\tau}}}

\newcommand{\Ds} {{\mathcal D}}

\newcommand{\Fs} {{\mathcal F}}

\newcommand{\Ls} {{\mathcal L}}

\title[Contiditioned first passage times]{The first passage time of a stable process conditioned to not overshoot}
\author{Fernando Cordero}
\address{Faculty of Technology, University of Bielefeld, Universit\"{a}tsstr. 25, 33615 Bielefeld, Germany}
\email{fcordero@techfak.uni-bielefeld.de}
\thanks{This work was partially done during the Post-Doctoral position held by the author at the Faculty of Mathematics of the University of Vienna. The author gratefully acknowledges financial support from the European Research Council (ERC) under grant agreement No.~247033.}
\date{\today}%
\begin{document}

\begin{abstract}
Consider a stable Lévy process $X=(X_t,t\geq 0)$ and let $\tx$, for $x>0$, denote the 
first passage time of $X$ above the level $x$. In this work, we give an alternative proof of the absolute continuity of the law of $\tx$
and we obtain a new expression for its density function. Our constructive approach provides a new insight into the study of the law of $\tx$. 
The random variable $\txz$, defined as the limit of $\tx$ when the corresponding overshoot tends to $0$, plays an important role in obtaining these results. 
Moreover, we establish a relation between the random variable $\txz$ and the dual process conditioned to die 
at $0$. This relation allows us to link the expression of the density function of the law of $\tx$ presented in this paper to the already known results on this topic.
\end{abstract}
\maketitle


\section{Introduction}\label{1}
Stable Lévy processes have been intensively used as models in many different fields such as finance, physics, hydrology and biology. 
One reason for this is that the one-dimensional distributions of the stable processes are heavy tailed and stable under addition. These properties are suitable, especially in modelling random noise and uncertain errors (see \cite{Bosc}, \cite{Mid} and \cite{Shni}). Another important feature of the stable processes is the scaling property, which is often observed in financial time series (see e.g. \cite{Bapr} and \cite{Man}).

In this work, we consider a stable Lévy process $X=(X_t, t\geq 0)$ with index $\alpha\in(0,2)$ and we study the law of its first passage time above a positive level $x$, i.e. $\tx=\inf\{s>0: X_s\geq x\}$. More precisely, we are interested in the absolute continuity of this law and in the representations of its density function. 

The law of $\tx$ is related to the law of the past supremum $S_t=\sup\{X_s: s\in[0,t]\}$ by means of the identity in law:
\begin{equation}\label{is}
\tx\overset{(d)}{=}x^\alpha\, {S_1}^{-\alpha},
\end{equation}
which follows from the scaling property. The study of these random variables is part of the theory of fluctuations, and an extensive literature on this topic is available. In particular, it is well known that the law of $S_1$ is absolutely continuous with respect to the Lebesgue measure. A proof of this, 
as well as an integral representation for the density function $f_{S_1}$, can be found in \cite[Theorem 4, p. 168]{Dar} 
in the case of symmetric stable processes, and in \cite[Theorem 1, p. 422]{Hey} in the case of general stable processes. 
These results together with \eqref{is} imply that the law of $\tx$ is also absolutely continuous and that its density function is given by
\begin{equation}\label{isd}
f_{\tx}(t)=\frac{1}{\alpha}\,x\, t^{-\frac{1}{\alpha}-1}\,f_{S_1}(x\,t^{-\frac{1}{\alpha}}).
\end{equation}
Additionally, in some particular cases, explicit series representations for $f_{\tx}$ can be obtained. For example, when $X$ has no positive jumps, the density function of $X_t$, denoted by $p_t$, and $f_{\tx}$ are related through the Kendall's identity (see e.g. \cite[p. 190]{Ber} or \cite {Bobu}), and the known series representations for $p_t$ (see \cite[pp. 87-89]{Sat}) lead to series representations for $f_{\tx}$.
On the other hand, when $X$ has no negative jumps, infinite series representations for $f_{\tx}$ can be deduced from those given in \cite[Theorem 1]{bedape} for $f_{S_1}$. Recently, in \cite{Kuzne} Kuznetsov obtains series representations and asymptotic expansions for the density function of the supremum for a large class of stable processes (see also \cite{Kupa}). 

The first goal of this paper is to give an alternative proof of the absolute continuity of the law of $\tx$ and to provide a new representation for the density function $f_{\tx}$ in terms of the distribution of $\txz$, defined below. In a second step, we aim to relate the latter to other classical quantities in fluctuation theory and, in particular, to infer from our results the already known density representations for $f_{S_1}$ obtained by Doney and Savov in \cite{Dosa} and Chaumont in \cite{Chaum4} (see also \cite{Alchau}).  

Our approach is based on the following observation. Note that, using the scaling property, the probability
$P(t-\varepsilon<\tx[1]\leq t)$ equals to $P\left(\tx[y_t]\leq1,\,\tx[y_{t-\varepsilon}]>1\right),$
where $y_u=u^{-1/\alpha}$. As an application of the strong Markov property at the time $\tx[y_t]$, we see that
$$P(t-\varepsilon<\tx[1]\leq t)=E\left[1_{\{\tx[y_t]\leq1,\,K_{y_t}\leq \,y_{t-\varepsilon}-y_t \}}\,\Phi\left(y_{t-\varepsilon}-y_t-K_{y_t},1-\tx[y_t]\right)\right],$$
where $\Phi(z,r)=P(\tx[z]>r)$, and $K_y=X_{\tx[y]}-y$ is the overshoot of $X$ at time $\tx[y]$. Informally, dividing by $\varepsilon$ in the previous identity and taking the limit when $\varepsilon$ tends to $0$, we can expect to find a relation between the absolute continuity of the law of $T_1$ and the asymptotic properties of the joint law of $\left(\tx[y_t],K_{y_t}\right)$ when $K_{y_t}$ is forced to converge to $0$. The latter problem is studied in \cite{FC} as follows. First, for each $\varepsilon>0$, the random variable $\tx^{\varepsilon}$ is defined as the random variable $\tx$ given that the corresponding overshoot $K_x$ is smaller than $\varepsilon$. Then, it is proven that, when $\varepsilon$ tends to $0$, the random variables $\tx^{\varepsilon}$ converge in distribution to a random variable denoted by $\txz$. Finally, the law of $\txz$ is characterized in terms of the law of $\tx$. 


The organization of the paper is as follows. In Section 2, we start by giving some definitions and recalling some preliminary facts about stable 
processes. Then, we give the definition of the random variable $\txz$ as well as some results concerning the characterization of its law.

In Section 3, we establish the asymptotic behaviour around $0$ of the distribution function of $\txz$, based on the tail distribution of the variable $S_1$ .

In Section 4, we show a simple relation between the Mellin transforms of the variables $\tx$ and $\txz$. Using this relation, we give a direct proof 
of the absolute continuity of the law of $\tx$ and simultaneously, we obtain the desired representation for its density function.

In Section 5, we look at the process $(X_s, 0\leq s\leq t)$ on $\{t<\tx\}$ conditioned on the event that the overshoot at time $\tx$ is smaller than $\varepsilon$, and then, we consider 
the limit when $\varepsilon$ tends to $0+$. This approach provides a relation between the asymptotic law and the law of the dual process killed upon leaving $[0,\infty)$ and conditioned to die at $0$. As a consequence of this relation and the expression of the density function $f_{\tx}$ obtained in Section 3, we retrieve both,  
Chaumont and Doney and Savov representations of the density function $f_{S_1}$.


\section{Preliminaries}\label{2}
In this section, we provide definitions, notations and already known results in the theory of stable Lévy processes which are used along the paper.
\subsection{Definitions and Notations}\label{2.1}
Let $\Ds=\Ds([0,\infty))$ be the space of càdlàg trajectories $\omega:[0,\infty)\rightarrow\Rb\cup\{\infty\}$. The lifetime of a trajectory $\omega\in \Ds$ is defined as $\varsigma(\omega)=\inf\{s\geq 0:\,\omega(s)=\infty\}$, with the usual convention $\inf\{\emptyset\}=\infty$. The space $\Ds$ is endowed with the Skorokhod's topology. We denote by $X=(X_t, t\geq 0)$ the coordinate process on $\Ds$ and by $(\Fs_t,t\geq 0)$ the natural filtration generated by $X$.

Let $P$ be the law on $\Ds$ of a stable Lévy process with index $\alpha\in (0,2)$. It means that, under $P$, the coordinate process $X$ has stationary and independent increments and satisfies the scaling property 
$$(X_{at}, t\geq 0)\overset{(d)}{=}(a^{1/\alpha}\,X_{t}, t\geq 0),\quad\textrm{for all } a>0.$$ 
In particular, under $P$, the process $X$ starts from $0$, i.e. $P(X_0=0)=1$. Moreover, for each $x\in\Rb$, we denote by $P_x$ the law of the canonical process starting from $x$, i.e. the law of $X+x$ under $P$. We write $P=P_0$.

The characteristic exponent of $X$, $\psi(\lambda)=-\ln\left(E[\exp(i\lambda X_1)]\right)$, has the form (see \cite{Ber}, \cite{zolo1} and \cite{zolo})
\begin{equation}\label{cf}
\psi(\lambda)=c{|\lambda|}^\alpha\left(1-i\sgn{(\lambda)}\tan\left(\pi\alpha(2\rho-1)/2\right)\right),
\end{equation}
where $c>0$ is a constant and $\rho=P(X_1>0)$ is the positivity parameter. For the sake of simplicity, we exclude the case when $\alpha=1$ and $\rho\neq 1/2$, i.e. we assume that $X$ is not an asymmetric Cauchy process.

Let $T_A$ denote the entrance time of a set $A\subset\Rb$ of $X$:
$$T_A=\inf\{s\geq 0: X_s\in A\}.$$
When the set $A$ has the form $[x,\infty)$ for some $x>0$, we simply write 
$\tx$ instead of $T_{[x,\infty)}$ for the first passage time of $X$ above the level $x$. We denote the overshoot at time $\tx$ by $K_x$, i.e. $K_x=X_{\tx}-x.$

We denote by $\dP$ the law of the dual process $\dX=-X$ under $P$ and, for $x>0$, by $\dP_x$ the law of the dual process starting at $x$, i.e. the law of
$x+\dX$ under $P$.

For $x>0$, $Q_x$ denotes the law on $\Ds$ of the process $(X,P_x)$ killed when it leaves $[0,\infty)$, i.e.
$$Q_x\left(\Lambda, t<\varsigma\right)=P_x\left(\Lambda, t<T_{(-\infty,0)}\right),\quad t\geq 0, \Lambda\in\Fs_t.$$
The process $(X,Q_x)$ is a Markov process on $(0,\infty)$. Moreover, the absolute continuity of the law of $X_t$ clearly implies the absolute continuity of the semigroup of $(X,Q_x)$. We denote the corresponding densities by $(q_t,t\geq 0)$.
In the same way, but replacing $X$ by its dual $\dX$, we define, for $x>0$, the measure $\dQ_x$ and the corresponding semigroup densities $(\dq_t,t\geq 0)$. 

Let $S_t$ and $I_t$ denote the past supremum and the past infimum of $X$ respectively, that is, for all $t\geq 0$,
$$S_t=\sup\{X_s:\,0\leq s\leq t\}\quad\textrm{ and }\quad I_t=\inf\{X_s:\,0\leq s\leq t\}.$$

It is well known that the reflected processes $R=S-X$ and $\dR=X-I$ verify the strong Markov property (see \cite{Bin}). Since $0$ is regular for 
$(0,\infty)$ with respect to $X$, we can define a local time at $0$ of $R$ which we denote by $L$.  Let $\tau$ be 
the right continuous inverse of the local time $L$. Then, $\tau$
is a stable subordinator of index $\rho$ (see \cite[Lemma 1, Chap. VIII]{Ber}). Similarly, but replacing $R$ by $\dR$, we define $\dL$ the local time at $0$ of $\dR$, and $\dtau$ its right continuous inverse, which is a stable subordinator of index $1-\rho$. The local times $L$ and $\dL$ are unique up to multiplicative constant, and we choose a normalization, so that the Laplace exponents of $\tau$ and $\dtau$ are given by $\phi(q)=q^\rho$ and $\dphi(q)=q^{1-\rho}$.

The Itô measures of excursions away from $0$ of the processes $R$ and $\dR$ are denoted by $n$ and $\dn$. These measures are Markovian, and their corresponding semigroups are given by $(\dq_t,t\geq 0)$ and $(q_t,t\geq 0)$. We denote by $(r_t,t>0 )$ and $(\dr_t,t>0 )$ 
the densities of the entrance laws of the reflected excursions at the maximum and at the minimum, i.e. for $t>0$,
$$r_t(x)=n\left(X_t\in dx,\, t<\varsigma\right)/dx\quad\textrm{and}\quad \dr_t(x)=\dn\left(X_t\in dx,\, t<\varsigma\right)/dx.$$
For $t>0$, $P^{(m,t)}$ denotes the law of the stable meander of length $t$, that is, the image of $n(\cdot|\varsigma>t)$ under the mapping $w\mapsto (w(s),0\leq s\leq t)$. The law of the dual stable meander of length $t$, $\dP^{(m,t)},$ is defined in the same way replacing $n$ by $\dn$.

In the remainder of the paper, we assume in addition that $X$ is not a subordinator and that $X$ is not spectrally negative. Since we also exclude the asymmetric Cauchy case, we are restricted to the following situations: 
\begin{equation}\label{calrho}
 (0<\alpha<1\,\wedge\, 0\neq\rho\neq1)\vee(\alpha=1 \,\wedge\,\rho=1/2)\vee(1<\alpha<2\,\wedge\,\alpha\rho<1).
\end{equation}

\subsection{\texorpdfstring{Representations of the Density Function $f_{S_1}$}{}}\label{2.2}
In this work, we are interested in obtaining new representations of the density function of the first passage time above the level $x$. From \eqref{isd}, this problem is equivalent to obtaining representations of the density function of the past supremum at time $1$. In this section, we recall two known representations of the latter density function.

The following expression is provided by Doney and Savov in \cite[Lemma 8]{Dosa}:
\begin{equation}\label{ci2}
f_{S_1}(x)=\frac{\sin(\pi\rho)}{\pi}\int\limits_0^1 \frac{\dpms_s(x)}{s^{1-\rho}\,(1-s)^\rho}ds,\quad x>0,
\end{equation}
where $\dpmt_t$ is the density function of the terminal value of the dual stable meander of length $t$. 
A similar result is obtained by Chaumont in \cite{Chaum4}, where he proves that the density function $f_{S_1}$ admits the following representation:
\begin{equation}\label{ci}
f_{S_1}(x)=\frac{1}{\Gamma(1-\rho)}\int\limits_0^1 \frac{\dr_s(x)}{(1-s)^\rho}ds,\quad x>0.
\end{equation}
Both expressions are equivalent, and we can pass from one 
to the other by using the elementary identity
\begin{equation}\label{frtpm}
 \dr_t(x)=\frac{1}{\Gamma(\rho)}\,t^{\rho-1}\,\dpmt_t(x).
\end{equation}
These representations are obtained by using the excursion theory for Lévy processes.
\subsection{\texorpdfstring{The Asymptotic Variable $\txz$}{}}In this section, we recall the definition and the main properties of the random variable $\txz$, which was introduced in \cite{FC} as the first passage time above the level $x$, given that the corresponding overshoot is $0$. The main results of this paper involve the law of $\txz$ and its properties. In particular, this random variable plays a key role in obtaining a new representation for the density function $f_{\tx}$. 

First, for each $x>0$ and $\varepsilon>0$, we introduce the random variable $\txe$ whose law is given by
$$P(\txe\in\cdot\,\,)=P(\tx\in\cdot\,\,\,|\,K_x\leq \varepsilon).$$
In \cite[Theorem 2]{FC} it is proven that, for each $x>0$, the family of random variables ${\{\txe\}}_{\varepsilon>0}$ converges in law as $\varepsilon$ tends to $0+$. The limit is denoted by $\txz$ and its law is given by 
$$P\left(\,\txz\leq\, t\,\right)=\frac{\sin (\pi\rho)}{k_0\,\pi\rho}\,x^{-\alpha\rho}E\left[\frac{\tx \,\,1_{\{\tx\leq\, t\}}}{{\left(t-\tx\right)}^{1-\rho}}\right].$$ 
An alternative characterization of the distribution of $\txz$ is given by its Laplace transform (see \cite[Proposition 5]{FC}):
\begin{equation}\label{eq1}
 E\left[\exp(-\lambda\txz\;)\right]=\frac{1}{k_0 \Gamma(1-\rho)\,\alpha\rho}\,x^{1-\alpha\rho}\,\lambda^{-\rho}\,f_\lambda(x),
\end{equation}
where $f_\lambda$ is the density function of the law of $\sel$ and $\ef_\lambda$ is an exponential random variable with parameter $\lambda$, independent of $X$.
Additionally, from the scaling property of $X$ it is straightforward to prove that
$$\txz[ax]\overset{(d)}{=}a^\alpha\txz.$$
\begin{rmq}
The aforementioned results concerning the distribution of $\txz$ are stated in \cite{FC} for stable processes with index $\alpha\in(1,2)$ such that $\alpha\rho<1$. However, these results remain true, without any modification on their proofs, under the weaker condition \eqref{calrho}. Indeed, the only restriction in \cite{FC} is the applicability of \cite[Theorem 3a and 4a]{Bin} (see \eqref{a1}, \eqref{a2}). Both results referring to stable processes, the first one needs that $|X|$ is not a subordinator and the second one that $X$ has positive jumps.
\end{rmq}

\subsection{Some Asymptotic Probabilities}
The law of the overshoot $K_x$ is well known in the literature. In \cite{Ray}, Ray gives an expression for its density function in the symmetric case. In \cite{Bin}, Bingham generalizes 
this result to the case when $X$ has positive jumps (see also \cite{yayayor2} and \cite{FC}). From these results, we have that
\begin{equation}\label{lxtx}
X_{\tx}    \overset{(d)}{=}  x\, \beta_{{\alpha\rho},1-{\alpha\rho}}^{-1},
\end{equation}
where $\beta_{{\alpha\rho},1-{\alpha\rho}}$ is a beta random variable with parameters $\alpha\rho$ and $1-{\alpha\rho}$. As a consequence (see \cite[Lemma 2]{FC}), the 
asymptotic behaviour of the distribution function of $K_x$ around $0+$ is given by
\begin{equation}\label{adkx}
P(K_x\leq h)\sim \frac{\sin (\pi\alpha\rho)}{\pi(1-\alpha\rho)}\;{\lr\frac{h}{x}\rr}^{1-\alpha\rho}\;\;\textrm{ as }\;\;h\rightarrow 0+.
\end{equation}
Other results that we use in this work concern the behaviour of the distribution of $S_1 $ around $0+$ and its asymptotic tail distribution, which are given by (\cite[Theorems 3a and 4a]{Bin}) 
\begin{equation}\label{a1}
 P(S_1\leq x)\sim k_0\, x^{\alpha\rho}\;\;\textrm{ as }\;\;x\rightarrow 0+,
\end{equation}
and
\begin{equation}\label{a2}
 P(S_1> x)\sim k_\infty\, x^{-\alpha}\;\;\textrm{ as }\;\;x\rightarrow \infty,
 \end{equation}
where $k_0$ and $k_\infty$ are explicit constants given by
$$k_0=\frac{1}{{c_1}^\rho\,\Gamma(1-\rho)\Gamma(1+\alpha\rho)},\quad k_\infty=c_1\,\Gamma(\alpha)\,\frac{\sin(\pi\alpha\rho)}{\pi},$$
and $c_1=c\,\left|\sec\left(\pi\alpha(2\rho-1)/2\right)\right|$, where $c$ is the constant in \eqref{cf}.


\section{\texorpdfstring{The Distribution of $\txz$ around $0+$}{}}\label{3}
In this section, we study the behaviour of the distribution function of $\txz$ around $0$. More precisely, we establish the rate of convergence to $0$ of $P(\txz\leq u)$ when $u$ tends to $0+$. 

The Tauberian theorem tells us that the behaviour around $0$ of the distribution function of $\txz$ is related with the behaviour of its Laplace transform at $\infty$ (see \cite[p. 10]{Ber}). Thus, thanks to \eqref{eq1}, it is enough to study the asymptotic behaviour of $f_\lambda(x)$ when $\lambda$ tends to $\infty$. In addition, the function $f_\lambda$ can be expressed as (see \cite[Proposition 3]{FC})
\begin{equation}\label{eq2}
f_\lambda(x)=\frac{\lambda\alpha}{x}E\left[\tx \exp(-\lambda\tx)\right],\quad x>0.
\end{equation}
For $x>0$, we introduce the function $\ell_x:(0,\infty)\rightarrow (0,\infty)$ defined by
$$\ell_x(\lambda)=\lambda f_\lambda(x),\quad \lambda>0.$$
\begin{lemme}\label{l1}
 For all $x>0$, the function $\ell_x$ is slowly varying at $\infty$, more precisely,
$$\lim\limits_{\lambda\rightarrow \infty}\ell_x (\lambda)=\frac{\alpha k_\infty}{x^{\alpha+1}}.$$
\end{lemme}
\begin{proof}
First note that using \eqref{eq2} and applying the Fubini's theorem,
\begin{align*}
 \ell_x(\lambda)&=-\frac{\alpha\lambda}{x}\,E\left[\,\int\limits_{\lambda \tx}^{\infty}e^{-v}(1-v)\, dv\right]=-\frac{\alpha\lambda}{x}\,\izi e^{-v}(1-v)\,P\left(\tx<\frac{v}{\lambda}\,\right)\,dv,
\end{align*}
and then, thanks to the scaling property (or directly using \eqref{is}), we obtain
\begin{align*}
 \ell_x(\lambda)&=-\frac{\alpha k_\infty}{x^{\alpha+1}}\,\izi e^{-v}(v-v^2)\,U\left(\frac{\lambda^{1/\alpha}x}{v^{1/\alpha}}\right)\,dv,
\end{align*}
where $U(y)=\frac{P(S_1>y)\,y^\alpha}{k_\infty}$. By \eqref{a2}, the function $U$ is bounded and $\lim\limits_{y\rightarrow \infty}U(y)=1$. Thus, 
the result follows as an application of the dominated convergence theorem.
\end{proof}
The next proposition provides the desired rate of convergence to $0$ of $P(\txz\leq u)$ when $u$ tends to $0$.
\begin{prop}\label{prop1}
For every $x>0$, we have
$$P(\txz\leq u)\sim \frac{k_\infty}{k_0}\,\frac{\sin(\pi\rho)}{\pi\rho^2 (1+\rho)}\,x^{-\alpha(1+\rho)}\, u^{1+\rho}\;\;\textrm{ as }\;\;u\rightarrow 0+.$$
\end{prop}
\begin{proof}
\noindent By \eqref{eq1}, the Laplace transform of $\txz$ can be expressed as
\begin{equation*}
 E\left[\exp(-\lambda\txz\;)\right]=\frac{1}{k_0 \Gamma(1-\rho) \alpha\rho}\,x^{1-\alpha\rho}\,\lambda^{-{(1+\rho)}}\,\ell_x(\lambda).
\end{equation*}
\noindent So, according to Lemma \ref{l1}, we obtain
\begin{equation*}
 E\left[\exp(-\lambda\txz\;)\right]\sim\frac{k_\infty}{k_0}\,\frac{x^{-\alpha(1+\rho)}}{\rho\,\Gamma(1-\rho)}\,\lambda^{-{(1+\rho)}}\;\;\textrm{ as }\;\;\lambda\rightarrow \infty.
\end{equation*}
The result is a consequence of the Tauberian theorem (see \cite[p. 10]{Ber}).
\end{proof}
\begin{rmq}
By the scaling property, \eqref{a2} can be rephrased in terms of the distribution of $\tx$ as follows:
$$P(\tx\leq u)\sim k_\infty\, x^{-\alpha}\,u\;\;\textrm{ as }\;\;u\rightarrow 0+.$$
In particular, we note that $P(\txz\leq u)$ converges faster to $0$ when $u$ tends to $0$ than $P(\tx\leq u)$.
\end{rmq}


\section{\texorpdfstring{Absolute Continuity and Density Representation of the Law of $\tx$}{}}\label{4}
The purpose of this section is to provide an alternative proof for the absolute continuity of the law of $\tx$ and to obtain
a new representation for its density function $f_{\tx}$ in terms of the law of the random variable $\txz$.
\subsection{Mellin Transforms}
Note that, from \eqref{a1}, \eqref{a2} and the relation \eqref{is}, we get that, for all $\gamma\in(-1,\rho)$, $E\left[{\tx}^\gamma\,\right]<\infty.$
On the other hand, by Proposition \ref{prop1}, we have that for all $\beta\in[0,1+\rho)$, $E\left[{(\txz)}^{-\beta}\,\right]<\infty.$
In fact, there is a simple relation between the moments of $\tx$ and those of $\txz$, which is stated in the next proposition.
\begin{prop}\label{propm}
For all $x>0$ and $\beta\in\Cb$ with $\re(\beta)\in(0,1+\rho)$, we have
\begin{equation}\label{mom}
 E\left[{\left(\frac{1}{\txz}\right)}^\beta\,\right]=\frac{1}{k_0 \rho\, B(\beta, 1-\rho)}\,x^{-\alpha\rho}\,E\left[{\left(\frac{1}{\tx}\right)}^{\beta-\rho}\,\right],
\end{equation}                                     
where $B(\cdot,\cdot)$ is the beta function. In particular, we have                              
$$E\left[{\left(\frac{1}{\txz}\right)}^\rho\,\right]=\frac{\sin(\pi\rho)}{k_0\pi }\, x^{-\alpha\rho}.$$
\end{prop}
\begin{proof}
By the scaling property of $\txz$ and $T_x$, it suffices to prove the result for $x=1$. Fix $\beta\in\Cb$ with $\re(\beta)\in(0,1+\rho)$. By using the identity
$x^{-\beta}=\frac{1}{\Gamma(\beta)}\izi e^{-xs}s^{\beta-1}ds,$ 
we obtain
$$E\left[{\left(\frac{1}{\txz[1]}\right)}^\beta\,\right]=\frac{1}{\Gamma(\beta)}\izi E\lc e^{-s\txz[1]}\rc\,s^{\beta-1}\,ds.$$
Thus, using \eqref{eq1} and \eqref{eq2}, we deduce from the previous identity that
\begin{align*}
 E\left[{\left(\frac{1}{\txz[1]}\right)}^\beta\,\right]&=\frac{1}{k_0\rho\Gamma(\beta)\Gamma(1-\rho)}\izi E\lc T_1 e^{-s T_1}\rc\,s^{\beta-\rho}\,ds\\
 &=\frac{1}{k_0\rho\,\Gamma(\beta)\Gamma(1-\rho)} E\lc T_1\izi e^{-s T_1} s^{\beta-\rho}\,ds\rc\quad\textrm{(Fubini's theorem)}\\
 &=\frac{\Gamma(\beta-\rho+1)}{k_0\rho\,\Gamma(\beta)\Gamma(1-\rho)} E\lc T_1^{\rho-\beta}\rc,
 \end{align*}
which proves the result.
\end{proof}
\subsection{\texorpdfstring{Density Representation of the Law of $\tx$}{}}
In the next theorem, we use the relation between the moments of $\tx$ and the moments of $\txz$ stated in Proposition \ref{propm} in order to obtain a new representation for the law of $\tx$. More precisely, by an inversion method, we deduce from \eqref{mom} the absolute continuity of this law as well as an expression for its density function.
\begin{thm}\label{thmactx}
For all $x>0$, the law of $\tx$ is absolutely continuous with respect to the Lebesgue measure on $[0,\infty)$ and its density function is given by
$$f_{\tx}(t)=k_0\,\rho\,\frac{x^{\alpha\rho}}{t}\,E\lc\frac{1_{\{\txz\leq t\}}}{{\lr t-\txz\rr}^\rho}\rc,\quad t>0.$$
\end{thm}
\begin{proof}
Recall that the moments of a beta random variable of parameters $a,b>0$ are given by
 $$E\lc\beta_{a,b}^p\rc=\frac{\Gamma(a+p)\,\Gamma(a+b)}{\Gamma(a+b+p)\,\Gamma(a)},\quad \re(p)>-a.$$
Then, the relation between the moments of $\tx$ and $\txz$ obtained in Proposition \ref{propm} can be expressed as
\begin{equation}\label{imtx}
E \lc  {\tx}^{-p}\rc = \frac{E\lc{\lr\frac{1}{\txz}\rr}^{\rho}\,{\lr\frac{\txz}{\beta_{\rho,1-\rho}}\rr}^{-p} \rc}{E \lc {\lr\frac{1}{\txz}\rr}^{\rho}\rc},\quad \re(p)\in(-\rho,1), 
\end{equation}
where $\beta_{\rho,1-\rho}$ is a beta random variable of parameters $\rho$ and $1-\rho$ independent of $\txz$. Denote $\mu_x$ the law of $\tx$ and $\nu_x$ the
probability measure on $[0,\infty)$ given by
$$\nu_x\left([s,t)\right)=\frac{E\lc{\lr\frac{1}{\txz}\rr}^{\rho}\,1_{\left\{\beta_{\rho,1-\rho}^{-1}\,\txz\,\in\,[s,t)\right\}} \rc}{E \lc {\lr\frac{1}{\txz}\rr}^{\rho}\rc},\quad 0\leq s< t.$$
The identity \eqref{imtx} says that the Mellin transforms of the measures $\mu_x$ and $\nu_x$ coincide on the strip 
$\{z\in\Cb: \re(z)\in(0,1+\rho)\}$, which implies that they are equal. In particular, we deduce that
\begin{equation}\label{idtx}
 P(\tx\leq t)=k_0\,\rho\, x^{\alpha\rho}\,\izt \frac{E\lc{\lr\frac{1}{\txz}\rr}^{\rho}\,1_{\left\{\txz\leq st\right\}}\rc}{ s^{1-\rho}{(1-s)}^{\rho}}\,ds.
\end{equation}
Note that using Fubini's theorem and the change of variable $u=\txz/s$, we get
\begin{align*}
\izt \frac{E\lc{\lr\frac{1}{\txz}\rr}^{\rho}\,1_{\left\{\txz\leq st\right\}}\rc}{ s^{1-\rho}{(1-s)}^{\rho}}\,ds&= E\lc{\lr\frac{1}{\txz}\rr}^{\rho}\,1_{\left\{\txz\leq t\right\}}\int\limits_{\txz/t}^{1} s^{\rho-1}{(1-s)}^{-\rho}\,ds\rc\\
&=E\lc 1_{\left\{\txz\leq t\right\}}\int\limits_{\txz}^{t} \frac{1}{u \,{\lr u-\txz\rr}^{\rho}}\,ds\rc.
\end{align*}
Applying Fubini's theorem to the term in the right-hand side and plugging the resulting expression in \eqref{idtx}, we obtain
$$P(\tx\leq t)=k_0\,\rho\, x^{\alpha\rho}\izt\frac{1}{u}\,E\lc\frac{1_{\{\txz\leq u\}}}{{\lr u-\txz\rr}^\rho}\rc\, du.$$
This concludes the proof.
 \end{proof}
\begin{rmq}
An alternative an shorter proof of the previous result follows from the expression of the Laplace transform of $\txz$ given by \eqref{eq1} and \eqref{eq2}. Indeed, from these identities, we deduce that
$$E\left[\tx \exp(-\lambda\tx)\right]=k_0\,\rho\, x^{\alpha\rho}\,\Gamma(1-\rho)\,\lambda^{\rho-1}\,E\left[ \exp(-\lambda\txz)\right].$$
Note that, the left-hand side is the Laplace transform of the measure $tP(\tx\in dt)$. Since $\Gamma(1-\rho)\lambda^{\rho-1}$ is the Laplace transform of $t^{-\rho}$, and the Laplace transform of a convolution is the product of the Laplace transforms, the right-hand side is the Laplace transform of the function
$$k_0\,\rho\, x^{\alpha\rho}\int\limits_{[0,t]}(t-s)^{-\rho}P(\txz\in dt)=k_0\,\rho\, x^{\alpha\rho}\,E\lc\frac{1_{\{\txz\leq t\}}}{{\lr t-\txz\rr}^\rho}\rc,$$
and the result follows.

The advantage of the proof with the Mellin transforms is that it also tell us that the law of $T_x$ under $P$ equals the law of $\beta_{\rho,1-\rho}^{-1}\txz$ under $P^{x,\rho}$, where $P^{x,\rho}$ is the probability measure defined by
$$P^{x,\rho}(A)=\frac{E\left[(\txz)^{-\rho}\,1_{A}\right]}{E\left[(\txz)^{-\rho}\right]}.$$
\end{rmq}


\section{\texorpdfstring{Properties and Representations of the Law of $\txz$}{}}\label{5}
By definition, the construction of the random variable $\txz$ uses only information of the joint law of $(\tx,K_x)$. In this section, we examine
the problem from the point of view of the process. This means that we study the asymptotic behaviour of the process $(X_s, 0\leq s\leq t)$ on $\{t<\tx\}$ given that $\{K_x\leq \varepsilon\}$ when $\varepsilon$ tends to $0+$. Using this approach, we provide a relation between the corresponding limit process and the dual process conditioned to die at $0$. After that, we use this relation to show that the law of $\txz$ is absolutely continuous with respect to the Lebesgue measure and we exhibit two expressions for its density function. The connections of our results with the existent literature are established in Section \ref{5.3}. In particular, the identities \ref{ci2} and \ref{ci} are deduced from the density representation provided in Theorem \ref{thmactx}. Finally, in the last section, we show an identity in law, similar to \eqref{is}, involving the random variable $\txz$. 
\subsection{\texorpdfstring{The Process Conditioned to Die at $0$}{}}
In this section, we recall the definition of the process conditioned to die at $0$ in the particular case of stable processes (see \cite{Chaum3}). For a more 
general and detailed treatment of this subject see \cite{Chaum2}. 

According to the results of \cite{Silve}, the excessive version of the potential density of the subordinator $(-X_{\dtau_\ell},\ell\geq 0)$ is 
harmonic for the process $(X,Q_x)$, $x>0$. Moreover, the function $h$ defined by
$$h(x)=E\lc\izi 1_{\{I_t\geq -x\}}d\dL_t\rc, \quad x\geq 0,$$
is invariant for this process and the potential density of  $(-X_{\dtau_\ell},\ell\geq 0)$ is $h'$. For each $x>0$, we define $P_x^\searrow$ as the law of the 
$h$-process associated with the semigroup $(q_t,t\geq 0)$  and the function $h'$. That is,
$$P_x^\searrow(\Lambda, t<\varsigma)=\frac{1}{h'(x)}E_x^Q\lc h'(X_t)\, 1_\Lambda\, 1_{\{t<\varsigma\}} \rc,\quad t\geq 0,\Lambda\in\Fs_t.$$
This law is introduced in  \cite{Chaum2}, where it is demonstrated that this process approaches $0$ at its lifetime. In other words, the process $(X,P_x^\searrow)$ 
corresponds to the process $(X,Q_x)$ conditioned to die at $0$. Following the same procedure, but replacing the process $X$ by its dual $\dX$, we can define the function $\hd$ with the analogous properties of $h$. Consequently, for each 
$x>0$, we define the law $\dP_x^\searrow$ as the $h$-process associated with the function $\dhp$ and the semigroup $(\dq_t,t\geq 0)$.

In the stable case, the functions $h$ and $\hd$ have the following form (see \cite{Chaum3}):
\begin{equation}\label{hdh}
 h(x)=c_0 \,x^{\alpha(1-\rho)}\quad\textrm{ and }\quad \hd(x)=\dc_0\, x^{\alpha\rho}\quad x>0,
\end{equation}
where $c_0$ and $\dc_0$ are positive constants. A direct proof of these identities is given in \cite[Proposition 2.3]{yayayor2}. Here, we provide an alternative proof which allows to compute the constants $c_0$ and $\dc_0$ explicitly.
\begin{lemme}[{\cite[p. 384]{Chaum3}}]
The functions $h$ and $\hd$ are given by
 $$h(x)=\Gamma(\rho)\,\dk_0\,x^{\alpha(1-\rho)}\quad\textrm{and}\quad \hd(x)=\Gamma(1-\rho)\,k_0\, x^{\alpha\rho}\quad x>0,$$
 where the constant $\dk_0$ is obtained by replacing $\rho$ by $1-\rho$ in the definition of $k_0$.
\end{lemme}
\begin{proof}
 We give the proof for $\hd$. The proof for $h$ follows the same lines.
 
For $q>0$, we consider the function
 $$\hd^q(x)= \izi E\lc e^{-q\tau_\ell}\,1_{\{S_{\tau_\ell}\leq x\}}\rc \,d\ell,\quad x>0.$$
Note that $\hd^q(x)$ converges to $\hd(x)$ when $q$ tends to $0$. 
On the other hand, it is possible to prove that (see the proof of \cite[Theorem 18, Chap. VI]{Ber})
$$q\izi e^{-qt} P(\tx>t)dt=\phi(q)\,\hd^q(x)=q^\rho\,\hd^q(x),$$
and then, making the change of variable $u=qt$ and using \eqref{is}, we can express this identity as
$$\izi e^{-u}P\lr S_1<x{\lr \frac{q}{u}\rr}^{1/\alpha}\rr du= q^\rho\,\hd^q(x).$$
Dividing this by $q^\rho$, taking the limit when $q$ tends to $0+$ and using \eqref{a1}, we obtain the desired result.
\end{proof}

\subsection{\texorpdfstring{A Relation Between $\txz$ and the Process $(X,\dP_x^{\searrow})$}{}}\label{5.1}
We start with a technical lemma, which will be useful in order to study the limit when $\varepsilon$ tends to $0+$ of the process $(X_s, 0\leq s\leq t)$ on $\{t<\tx\}$ given that $\{K_x\leq \varepsilon\}$.
\begin{lemme}\label{exc}
For all $x,t>0$ we have
$$E\lc1_{\{\tx>t\}}{\lr x-X_t\rr}^{\alpha\rho-1}\rc\leq x^{\alpha\rho-1}.$$
\end{lemme}
\begin{proof}
Note that
\begin{align*}
 E&\lc1_{\{\tx>t\}}\,{\lr x-X_t\rr}^{\alpha\rho-1}\rc= \dE\lc 1_{\{T_{(-\infty,-x)}>t\}}\,{\lr x+X_t\rr}^{\alpha\rho-1}\rc\\
 &=\dE_x\lc 1_{\{T_{(-\infty,0)}>t\}}\,{X_t}^{\alpha\rho-1}\rc=\frac{1}{\dc}\,E_x^{\dQ}\lc 1_{\{\varsigma>t\}}\,\dhp(X_t)\rc.
\end{align*}
The result follows from the fact that the function $\dhp$ is $(\dq_t)$-excessive (see \cite{Chaum2} and \cite{Silve}).
\end{proof}
Now, we can state the main theorem of this section.
\begin{thm}\label{thmcr}
 For all $x,t>0$ and $\Lambda\in\Fs_t$, we have
 $$\lim_{\varepsilon\rightarrow 0+}P\lr\Lambda,\, \tx>t\,\arrowvert\, K_x\leq \varepsilon\rr= \dP_x^{\searrow}(\Lambda_x,\,\zeta>t),$$
 where $\Lambda_x=\{w:x-w\in \Lambda\}$.
\end{thm}
\begin{proof}
 Fix $x,t>0$ and $\Lambda\in\Fs_t$. Note first that
 $$P\lr\Lambda,\, \tx>t\,\arrowvert\, K_x\leq \varepsilon\rr=E\lc 1_{\Lambda}\, 1_{\{\tx>t\}}\, \frac{E[K_x\leq \varepsilon\,\arrowvert\, \Fs_t]}{P\lr K_x\leq \varepsilon\rr}\rc.$$
 Thus, using the Markov property, we obtain
 \begin{equation}\label{iamp}
 P\lr\Lambda,\, \tx>t\,\arrowvert\, K_x\leq \varepsilon\rr=E\lc 1_{\Lambda}\, 1_{\{\tx>t\}}\, \frac{\psi(x-X_t;\varepsilon)}{\psi(x;\varepsilon)}\rc,
 \end{equation}
 where $\psi(y;\varepsilon)=P\lr K_y\leq \varepsilon\rr$. 
 
Applying \eqref{adkx}, we get that, for all $y,z>0$,
\begin{equation}\label{cas}
 \frac{\psi(y;\varepsilon)}{\psi(z;\varepsilon)}\xrightarrow[\varepsilon\rightarrow 0+]{}{\lr\frac{y}{z}\rr}^{\alpha\rho-1}.
\end{equation}
Moreover, using the identity in law \eqref{lxtx}, it follows that for all $y>0$,
$$\psi(y;\varepsilon)=\frac{\sin(\pi\alpha\rho)}{\pi}\;\;y^{\alpha\rho}\varepsilon^{1-\alpha\rho}\id{0}{1}\frac{1}{(u\varepsilon+y)u^{\alpha\rho}}du.$$
The previous expression allows us to prove that, for every $y>0$ and $\varepsilon\leq z$,
\begin{equation}\label{dl1}
 \frac{\psi(y;\varepsilon)}{\psi(z;\varepsilon)}\leq{\lr\frac{y}{z}\rr}^{\alpha\rho}\frac{(z+\varepsilon)}{y}\leq 2\,{\lr\frac{y}{z}\rr}^{\alpha\rho-1}.
\end{equation}
Thanks to \eqref{cas}, \eqref{dl1} and Lemma \ref{exc}, we can take the limit in \eqref{iamp} when $\varepsilon$ tends to $0+$ and apply the dominated convergence theorem to deduce that
$$\lim_{\varepsilon\rightarrow 0+}P\lr\Lambda,\, \tx>t\,\arrowvert\, K_x\leq \varepsilon\rr=E\lc 1_{\Lambda}\, 1_{\{\tx>t\}}\, {\lr\frac{x-X_t}{x}\rr}^{\alpha\rho-1}\rc.$$
Note that the right-hand side can be written as
\begin{align*}
 E&\lc 1_{\Lambda}\, 1_{\{\tx>t\}}\, {\lr\frac{x-X_t}{x}\rr}^{\alpha\rho-1}\rc=\dE\lc 1_{\Lambda_0}\, 1_{\{T_{(-\infty,-x)}>t\}}\, {\lr\frac{x+X_t}{x}\rr}^{\alpha\rho-1}\rc\\
&=\dE_x\lc 1_{\Lambda_x}\, 1_{\{T_{(-\infty,0)}>t\}}\, {\lr\frac{X_t}{x}\rr}^{\alpha\rho-1}\rc=E_x^{\dQ}\lc 1_{\Lambda_x}\, 1_{\{\varsigma>t\}}\, {\lr\frac{X_t}{x}\rr}^{\alpha\rho-1}\rc.
\end{align*}
The result follows from the definition of $\dP_x^{\searrow}$.
\end{proof}
\begin{rmq}
The statement of the previous theorem can be equivalently rewritten as follows:
$$\lim_{\varepsilon\rightarrow 0+}\dP_x\lr\Lambda,\, T_{(-\infty,0)}>t\,\arrowvert\, X_{T_{(-\infty,0)}}\geq -\varepsilon\rr= \dP_x^{\searrow}(\Lambda,\,\zeta>t).$$
This result is very close to Proposition 3 in \cite{Chaum2} which states that for every $\beta>0$, we have
$$\lim_{\varepsilon\rightarrow 0+}\dP_x\lr\Lambda,\, T_{(-\infty,\beta)}>t\,\arrowvert\,I_{T_{(-\infty,0)}-}\leq \varepsilon\rr= \dP_x^{\searrow}(\Lambda,\,T_{(0,\beta)}>t),$$
Moreover, in the stable case, we can relax the condition $\beta>0$ in the following way. First, note that
$$\dP_x\lr\Lambda,\, T_{(-\infty,\varepsilon)}>t\,\arrowvert\,I_{T_{(-\infty,0)}-}\leq \varepsilon\rr=P\left(\Lambda_x, T_{x-\varepsilon}>t\arrowvert\,K_{x-\varepsilon}\leq \varepsilon\right).$$
Thus, following the lines of the proof of Theorem \ref{thmcr}, we can show that
$$\lim_{\varepsilon\rightarrow 0+}\dP_x\lr\Lambda,\, T_{(-\infty,\varepsilon)}>t\,\arrowvert\,I_{T_{(-\infty,0)}-}\leq \varepsilon\rr= \dP_x^{\searrow}(\Lambda,\,\zeta>t).$$
The main difference between the two results lies in the conditioning part. In the latter result, the process $(X,\dP_x)$ is conditioned to reach continuously the level $0$, while in our result the process $(X,\dP_x)$ is conditioned to leave continuously the interval $[0,\infty)$.
As a consequence of the two results, we can see the random variable $\txz$ as a first passage time above the level $x$, when the process $X$ is conditioned 
to cross the level $x$ continuously.
\end{rmq}
\begin{rmq}
Following the discussion of the previous remark, we see that Proposition 3 in \cite{Chaum2}, which holds in a more general framework, can be expressed in terms of overshoots as follows:
$$\lim_{\varepsilon\rightarrow 0+} P\left(\Lambda, T_{x-\beta}>t\arrowvert\,K_{x-\varepsilon}\leq \varepsilon\right)=\dP_x^{\searrow}(\Lambda_x,\,T_{(0,\beta)}>t),\quad x>\beta>0.$$

\end{rmq}

\begin{cor}\label{coril}
For all $x>0$, the law of $\txz$ is equal to the law of the lifetime $\zeta$ of $X$ under $\dP_x^{\searrow}$.
\end{cor}
\begin{proof}
Direct from Theorem \ref{thmcr}.
\end{proof}
\subsection{A Few Words on the Lamperti Transformation}\label{lt}
The process $(X,\dP_x^{\searrow})$ is a positive self-similar Makov process (PSSMP) of index $1/\alpha$. The underlying Lévy process $(\xi^{\searrow},\dPb)$ in the Lamperti representation is characterized in \cite[Corollary 3]{Cachau}. From the Lamperti  transformation and Corollary \ref{coril} we have that
\begin{equation}\label{lamp}
 (\txz,P)\overset{(d)}{=}(\zeta,\dP_x^{\searrow})\overset{(d)}{=}\left(x^\alpha\izi e^{\alpha\xi_s^{\searrow}}ds,\dPb\right).
\end{equation}
When the process $(X,P)$ is spectrally positive, i.e. $\alpha(1-\rho)=1$, the process $(\xi^{\searrow},\dPb)$ is spectrally negative (see \cite{Cachau}). Consequently, the results in \cite{Pat} can be applied to deduce that the law of $\txz$ is absolutely continuous and its density function $f_{\txz}$ infinitely continuously differentiable. Moreover, power series and contour integral representations for the latter density are in \cite{Pat} provided.

In the general case, we can deduce from \cite[Section 5.3]{Kupa} that the process $(\xi^{\searrow},\dPb)$ is a hypergeometric Lévy process with parameters $(0,\alpha(1-\rho),0,\alpha\rho)$. Therefore, we have that
\begin{equation}\label{lamp1}
\txz\overset{(d)}{=}x^\alpha\izi e^{-\alpha\xi_s}ds,
\end{equation}
where $\xi$ is a hypergeometric process with parameters $(1,\alpha\rho,1,\alpha(1-\rho))$. Let $\Ls$ denote the set of irrational numbers $x$, for which there exists a constant $b>1$ such that the inequality $|x-\frac{p}{q}|<b^{-q}$ is satisfied for infinitely many integers $p$ and $q$. According to \eqref{lamp1} and the results in \cite{Kupa}, if $\alpha\notin\Ls\cup\Qb$ we have that the law of $\txz$ is absolutely continuous and its density function admits the following convergent series expansion:
$$f_{\txz}(t)=x^{-\alpha(1+\rho)}t^{\rho}\sum\limits_{n=0}^{\infty}b_{m,n}\,\left(\frac{t}{x^\alpha}\right)^{m/\alpha+n},\quad\textrm{if }\alpha<1,$$
and
$$f_{\txz}(t)=x t^{-1-1/\alpha}\sum\limits_{m=0}^{\infty}\sum\limits_{n=0}^{\infty}c_{m,n}\,\left(\frac{t}{x^\alpha}\right)^{-m/\alpha-n},\quad\textrm{if }\alpha>1,$$
where the coefficients $\{b_{n,m}\}_{m,n\geq 0}$ and $\{c_{n,m}\}_{m,n\geq 0}$ are defined in \cite[pp. 125-126]{Kupa} using the notations $\beta=\widehat{\beta}=1$, $\gamma=\alpha\rho$ and $\widehat{\gamma}=\alpha(1-\rho)$.

\subsection{\texorpdfstring{Density Representation of the Law of $\txz$}{}}\label{5.2}
The following proposition states the absolute continuity of the law of $\txz$ and provides two representations for its density function.
\begin{prop}\label{actxo}
For all $x>0$, the law of $\txz$ is absolutely continuous and its density function $f_{\txz}$ is given by
\begin{equation}\label{ftxz1}
f_{\txz}(t)=\frac{1}{\alpha\rho\, k_0\,\Gamma(1-\rho)}\,x^{1-\alpha\rho}\,\dr_t(x), \quad t>0,
\end{equation}
or equivalently, by
\begin{equation}\label{ftxz2}
f_{\txz}(t)=\frac{\sin(\pi\rho)}{\pi\alpha\rho\, k_0}\,x^{1-\alpha\rho}\,t^{\rho-1}\,\dpmt_t(x), \quad t>0.
\end{equation}
\end{prop}

\begin{proof}
By Corollary \ref{coril} and the definition of $\dP_x^{\searrow}$, we have that
\begin{equation}\label{c4e1}
 P(\txz>t)=\dP_x^{\searrow}(\varsigma>t)=E_x^{\dQ}\lc\frac{\dhp(X_t)}{\dhp(x)}\,1_{\{\varsigma>t\}}\rc.
\end{equation}
On the other hand, using an expression of the potential under $\dn$ due to Silverstein (\cite[Eq. (3.3)]{Silve}), we have
$$\dhp(x)=\izi \dr_s(x)\,ds.$$
Plugging this identity in \eqref{c4e1} and applying Fubini's theorem, we obtain
$$P(\txz>t)=\frac{1}{\dhp(x)}\izi E_x^{\dQ}\lc \dr_s(X_t)\,1_{\{\varsigma>t\}}\rc\,ds.$$
Note that
$$E_x^{\dQ}\lc \dr_s(X_t)\,1_{\{\varsigma>t\}}\rc=\izi \dr_s(y)\, \dq_t(x,y)\,dy=\izi \dr_s(y)\, q_t(y,x)\,dy=\dr_{t+s}(x).$$
It follows that
$$P(\txz>t)=\frac{1}{\dhp(x)}\int_t^{\infty}\dr_u(x)\,du.$$
The absolute continuity of the law of $\txz$, as well as the identity \eqref{ftxz1}, is established. The identity \eqref{ftxz2} follows from \eqref{ftxz1} by using \eqref{frtpm}. 
\end{proof}

\subsection{Connections with the Existent Literature}\label{5.3}
As a corollary of the previous results, we obtain two representations for the density function $f_{\tx}$, which, by means of \eqref{isd}, are equivalent to \eqref{ci} and \eqref{ci2}. 
\begin{cor}[{\cite[Corollary 4]{Chaum4}} and {\cite[Lemma 8]{Dosa}}]
For all $x>0$, the density function of $\tx$ admits the following representations:
$$f_{\tx}(t)=\frac{1}{\alpha\,\Gamma(1-\rho)}\, \frac{x}{t}\,\int\limits_0^t \frac{1}{{(t-s)}^\rho}\,\dr_s(x)\,ds,\quad t>0,$$
and 
$$f_{\tx}(t)=\frac{\sin(\pi\rho)}{\alpha\pi}\, \frac{x}{t}\,\int\limits_0^t \frac{\dpms_s(x)}{s^{1-\rho}\,(t-s)^\rho}ds,\quad t>0.$$
\end{cor}
\begin{proof}
The first representation follows by plugging \eqref{ftxz1} in the expression of $f_{\tx}$ given in Theorem \ref{thmactx}.
The second representation is obtained from the first one by using \eqref{frtpm}.
\end{proof}
Doney and Savov provide in \cite{Dosa} the asymptotic behaviour of the density function $\dpm$ at infinity and at zero. These results together with Proposition \ref{actxo} 
allow us to give the analogous results for $f_{\txz}$.
\begin{cor}\label{cadtxz}
There is a constant $C_\infty>0$, such that
\begin{equation}\label{ftyxi}
f_{\txz}(t)\sim C_\infty\, x\, t^{-1-1/\alpha} \;\;\textrm{ as }\;\;t\rightarrow \infty.
\end{equation}
On the other hand, we have
\begin{equation}\label{ftyxz}
 f_{\txz}(t)\sim \frac{k_\infty}{k_0}\,\frac{\sin(\pi\rho)}{\pi\rho^2 }\,x^{-\alpha(1+\rho)}\, t^{\rho}\;\;\textrm{ as }\;\;t\rightarrow 0+,
\end{equation}
\end{cor}
\begin{proof}
By the scaling property and Proposition \ref{actxo}, we have
$$f_{\txz}(t)=\frac{\sin(\pi\rho)}{\pi\alpha\rho\, k_0}\,x^{1-\alpha\rho}\,t^{\rho-1-1/\alpha}\,\dpm(x\,t^{-1/\alpha}).$$
From \cite[Theorem 1]{Dosa}, we know that, there is a constant $C>0$, such that
$$\dpm(y)\sim C\, y^{\alpha\rho} \;\;\textrm{ as }\;\;y\rightarrow 0+,$$
and that
$$\dpm(y)\sim \frac{\alpha\, k_\infty}{\rho}\, y^{-(\alpha+1)} \;\;\textrm{ as }\;\;y\rightarrow \infty.$$
The results follow.
\end{proof}
\begin{rmq}
At the best of my knowledge, there is no closed form formula available in the literature for the constant $C$. However, from \eqref{lamp1} and \cite[Theorem 3]{Kupa}, we can deduce for $\alpha\notin\Qb$ that
$$f_{\txz}(t)\sim \bar{c} \,c_{0,0}\, x\, t^{-1-1/\alpha} \;\;\textrm{ as }\;\;t\rightarrow \infty,$$
where 
$$\bar{c}=c\,\frac{\left|\sec\left(\pi\alpha(2\rho-1)/2\right)\right|\,\pi}{2\sin(\pi\alpha/2)\Gamma(1+\alpha)},\quad\quad c_{0,0}=\frac{\Gamma(1+1/\alpha)}{\Gamma(1-\alpha\rho)\Gamma(\alpha(1-\rho))}M(1/\alpha),$$
and $M$ is a function defined in \cite[p. 121]{Kupa} with the help of some double gamma functions. Comparing with \eqref{ftyxi}, we see that, when $\alpha\notin\Qb$, we have $$C_\infty=\bar{c}\, c_{0,0}\quad\textrm{and}\quad C=\bar{c}\, c_{0,0}\frac{\pi\alpha\rho k_0}{\sin(\pi\rho)}.$$
\end{rmq}

\begin{rmq}
Note that \eqref{ftyxz} could not be deduced directly from Proposition \ref{prop1} since we do not know whether or not the density function $f_{\txz}$ is ultimately monotone.
\end{rmq}
\begin{rmq}
 From Theorem \ref{thmactx} and \eqref{ftyxz}, we infer that the density function $f_{\tx}$ is strictly positive in $(0,\infty)$. On the other hand, we know from \cite[Lemma 3]{Uri} and \cite[Theorem 1]{Chaum3} that the density function $\dpm$ is strictly positive. Consequently, the density function $f_{\txz}$ is also strictly positive. Therefore, the distributions of $\tx$ and $\txz$ are equivalent and
 $$\frac{d\mu_x^0}{d\mu_x}(t)=\frac{1}{k_0\, \rho\,x^{\alpha\rho}}\,f_{\txz}(t){\left(\frac{1}{t}\,\int\limits_{0}^t\frac{f_{\txz}(s)}{(t-s)^\rho} ds\right)}^{ -1},$$
 where $\mu_x^0(\cdot):=P(\txz\in\cdot)$ and $\mu_x(\cdot):=P(\tx\in\cdot)$.
 \end{rmq}

\begin{rmq}
The results of Corollary \ref{cadtxz} can be compared to the analogous results for $f_{\tx}$, which state that (see \cite[Remark 5]{Dosa})
$$f_{\tx}(t)\sim \rho\, k_0\, x^{\alpha\rho}\, t^{-(1+\rho)} \;\;\textrm{ as }\;\;t\rightarrow \infty,$$
and
$$f_{\tx}(t)\sim k_{\infty} \,x^{-\alpha}\;\;\textrm{ as }\;\;t\rightarrow 0+.$$
\end{rmq}
\subsection{An Identity in Law}\label{5.4} We end this work with an identity in law, which is in some sense, the analogue of \eqref{is} for $\txz$. Note first that from the asymptotic behaviour of $\dpm$ (see \cite[Theorem 1]{Dosa}), we have
$$0<\dE^{(m,1)}[X_1^{-\alpha\rho}]<\infty.$$
Consequently, we can define the probability measure $P^{(0)}$ on $\Fs_1$ by
$$P^{(0)}(\Lambda)=\frac{\dE^{(m,1)}\lc 1_\Lambda \,X_1^{-\alpha\rho}\rc}{\dE^{(m,1)}[X_1^{-\alpha\rho}]}.$$
\begin{cor}
For all $x>0$, the law of $\txz$ equals the law of $x^\alpha \,X_1^{-\alpha}$ under $P^{(0)}$.
\end{cor}
\begin{proof}
Let $g:[0,\infty)\rightarrow[0,\infty)$ be a measurable function. By Proposition \ref{actxo} and the scaling property, we have
$$E[g(\txz)]=\frac{\sin(\pi\rho)}{\pi\alpha\rho\, k_0}\,x^{1-\alpha\rho}\,\izi g(t)\, t^{\rho-1-1/\alpha}\,\dpm(x\,t^{-1/\alpha})\, dt.$$
Thus, making the change of variable $y=x\,t^{-1/\alpha}$, we get
$$E[g(\txz)]=\frac{\sin(\pi\rho)}{\pi\rho\, k_0}\,\izi g\left( x^\alpha\,y^{-\alpha}\right)\, y^{-\alpha\rho}\,\dpm(y)\, dy,$$
or equivalently
\begin{equation}\label{ilaux}
E[g(\txz)]=\frac{\sin(\pi\rho)}{\pi\rho\, k_0}\,\dE^{(m,1)}\lc  X_1^{-\alpha\rho}\,g\left( x^\alpha\,X_1^{-\alpha}\right)\rc.
 \end{equation}
Taking $g=1$ in the previous identity, we obtain
$$\dE^{(m,1)}\lc  X_1^{-\alpha\rho}\rc= \frac{\pi\rho\, k_0}{\sin(\pi\rho)}.$$
The result follows by plugging this expression in \eqref{ilaux}.
\end{proof}
\textbf{Acknowledgments.} I would like to thank the anonymous referee whose valuable suggestions and comments contributed to the quality of this version of the paper.

\bibliographystyle{acm}
\bibliography{reference}
\end{document}